\documentclass[12pt]{amsart}
\usepackage[utf8]{inputenc}
\usepackage{graphicx}

\usepackage{mathtools, bm}
\usepackage{amssymb, bm}
\usepackage{amsthm}
\usepackage{amsfonts}
\usepackage{stmaryrd}
\usepackage{hyperref, enumerate}
\usepackage{amsmath}
\usepackage[T1]{fontenc}
\usepackage{setspace}
\usepackage{dsfont}
\usepackage{array, color}
\usepackage{fancybox}
\usepackage{marvosym}
\usepackage{wasysym}
\usepackage{soul}
\usepackage{tikz}
\usetikzlibrary{matrix, arrows}

\newcommand{\Q}{\mathbb{Q}}

\newcommand{\Sgothique}{\mathfrak{S}}

\newtheorem{thm}{Theorem}[section]
\newtheorem{lemma}[thm]{Lemma}
\newtheorem{coro}[thm]{Corollary}
\newtheorem{dfn}[thm]{Definition}
\newtheorem{rmq}[thm]{Remark}
\newtheorem{prop}[thm]{Proposition}

\newtheorem{Criteria}[thm]{Criterion}
\newtheorem{Question}[thm]{Question}

\begin{document}
\title[]{On a local property of infinite Galois
extensions implying the Northcott property}
\author{Jonathan Jenvrin}
\address{Jonathan Jenvrin, Univ. Grenoble Alpes, CNRS, IF, 38000 Grenoble, France. \textit{E-mail adress :} \href{mailto:jonathan.jenvrin@univ-grenoble-alpes.fr}
{\texttt{jonathan.jenvrin@univ-grenoble-alpes.fr}}
}
\date{}
\maketitle

\begin{abstract}

In 2001, Bombieri and Zannier studied the Northcott property (N) for infinite Galois extensions of the rationals. In particular they provided a local property of the extensions that imply property (N). Later, Checcoli and Fehm demonstrated the existence of infinite extensions satisfying this local property. In this article, we establish two main results. First, we show that this local property, unlike property (N), is not preserved under finite extensions. Second, we show that, for an infinite Galois extension of $\mathbb{Q}$, such local property cannot be read from the Galois group. More precisely, we exhibit several profinite groups that can be realized over $\mathbb{Q}$ by fields with and also by fields without this local property.

\end{abstract}
\section{Introduction}

In this article, we let $\overline{\Q}$ denote a fixed algebraic closure of $\Q$. For an algebraic number $\alpha$, we denote by $h(\alpha)$ denotes its absolute logarithmic Weil height defined as
$$h(\alpha) = \frac{1}{d} \left( \log \left(|a| \prod_{i=1}^{d} \max\left(1, |\alpha_{i}|\right) \right) \right)$$
where $a$ is the leading coefficient of the minimal polynomial of $\alpha$ over $\mathbb{Z}$, and $\alpha_{1}, \dots, \alpha_{d}$ are the conjugates of $\alpha$ over $\Q$ (see \cite[Section 1]{BombieriGubler} for a more general definition).

Following Bombieri and Zannier \cite{bombierizanniernorthcottproperty}, we say that a set of algebraic numbers $S \subset \overline{\Q}$ has the \textit{Northcott property (N)} if the  set
\[ \{\alpha \in S \mid h(\alpha) \le M \} \]
is finite for every positive real number $M$.

Every number field is known to have property (N) due to the Northcott theorem. However, for infinite extensions of $\mathbb{Q}$, there is no general criterion to decide whether they possess property (N) in general.

We now discuss what is known on this problem. In \cite{bombierizanniernorthcottproperty} Bombieri and Zannier provided a first example of an infinite extension of $\Q$ with property (N): \( K^{(d)}_{ab} \), the compositum of all abelian extensions of degree at most \( d \) of a number field \( K \) \cite[Theorem 1]{bombierizanniernorthcottproperty}. Later, in \cite[Theorem 3]{widmernorthcottcriteria}, Widmer gives a criterion for an infinite extension of $\mathbb{Q}$ to have property (N), based on the growth of the relative discriminants of its finite subextensions (see Subsection \ref{Subsection On the Intersection of Three Criteria} for a precise statement).

In this article we investigate another criterion for the Northcott property $(N)$,  for infinite Galois extensions of the rationals, in relation to certain \textit{local conditions} of the extension. To state it we need to fix some notation. 
For a (possibly infinite) Galois extension $L/\mathbb{Q}$ and a rational prime number $p$, we define 
$$
\Sgothique_{p}(L)=\frac{\log p}{e_{p}(L)(p^{f_{p}(L)}+1)}
$$ 
where $e_{p}(L)$ and $f_{p}(L)$ are the (possibly infinite) ramification index and inertia degree of the extension $L/\mathbb{Q}$ at the prime $p$. These quantities depend only on $p$, since $L/\mathbb{Q}$ is Galois. We adopt the convention that $\Sgothique_{p}(L)=0$ if $L$ cannot be embedded in a finite extension of $\Q_p$. We denote by $\mathcal{P}$ the set of all rational prime numbers. We then set
$$
\Sgothique(L)= \sum_{p \in \mathcal{P}}\Sgothique_{p}(L).
$$
We say that a field $L\subseteq \overline{\Q}$ has \emph{property $(\Sgothique)$}, if $L/\Q$ is Galois and $\Sgothique(L)=+\infty$. 
The following criterion from  \cite[Remark p.8]{bombierizanniernorthcottproperty} is a straightforward consequence of \cite[Theorem 2]{bombierizanniernorthcottproperty}.
\begin{Criteria} \label{Criteria (S) imply (N)}
    Property $(\Sgothique)$ implies property (N).
\end{Criteria}
It is noteworthy that recent work of Dixit and Kala \cite[Theorem 1.1]{(B)fornotgaloisextension} inspired by the work of Bombieri and Zannier \cite[Theorem 2]{bombierizanniernorthcottproperty}, generalize this criterion to the non-Galois case, i.e. that their result give a criterion for an infinite extension $L$ over $\Q$, not necessarily Galois, to have property (N), and in the case where the extension $L/\Q$ is Galois, their result is exactly the one of Bombieri and Zannier.

While, as remarked in \cite[Remark p.8]{bombierizanniernorthcottproperty}  property  $(\Sgothique)$ holds true for finite Galois extensions of $\Q$ (see for instance Proposition \ref{numbers field case}), a natural question to consider is whether there exist infinite Galois extensions $L/\mathbb{Q}$ with property $(\Sgothique)$.

In \cite{checcolinorthcottproperty}, it was shown that such infinite Galois extensions do exist and, moreover, they can be constructed in a way so that they do not satisfy any other previously known criterion for the Northcott property (see \cite[Theorem 1.5]{checcolinorthcottproperty} and \cite[Theorem 1.2]{checcoli2024widmerscriterianorthcottproperty}), such as Widmer’s criterion (see \cite[Theorem 3]{widmernorthcottcriteria} and \cite[Theorem 8]{widmernewcriteria}) or the result of Bombieri and Zannier on $K_{ab}^{(d)}$ (see \cite[Theorem 2]{bombierizanniernorthcottproperty}).
%There are natural examples that satisfy property $(N)$ but do not satisfy $\Sgothique(L)=+\infty$, such as $\mathbb{Q}^{(2)}$ (the compositum of all quadratic fields over $\Q$, see \cite[Proposition 2.6]{checcolinorthcottproperty}).

It is known that property $(N)$ is preserved under finite extensions, as shown in \cite[Theorem 2.1]{property(N)preservedfiniteextension}. In \cite[Remark 2.8]{checcolinorthcottproperty} Checcoli and Fehm raised the following question.
\begin{Question} \label{preserved by finite extension ?}
Let $L/\mathbb{Q}$ be a Galois extension such that $\Sgothique(L)=+\infty$, and let $F/L$ be a finite extension with $F/\mathbb{Q}$ Galois. Is it true that $\Sgothique(F)=+\infty$?
\end{Question} 

In this article,  we provide a negative answer to this question, proving the following:
\begin{thm} \label{Theorem Property S does't preserved under finite extension}
    There exists an infinite totally real Galois extension $L/\Q$ such that $\mathfrak{S}(L)=+\infty$ and $\mathfrak{S}(L(i))<\infty$. In particular, property $(\mathfrak{S})$ is not preserved under finite extensions.
\end{thm}

Our proof strategy for Theorem \ref{Theorem Property S does't preserved under finite extension} relies on results, presented in Section \ref{prel-splitting}, concerning the natural density of certain subsets of rational primes in number fields (Subsection \ref{Subsection Auxiliary result density prime}) and the splitting of primes in certain finite Galois extensions (Subsection \ref{Subsection Grunwald Problem}). The field $L$ in Theorem \ref{Theorem Property S does't preserved under finite extension} is then constructed in Section \ref{answer to question about preservation under finite extension}.
 %Here, we will also adopt a group-theoretic perspective, examining the relationship between the Northcott property for a Galois extension and its Galois group.

Following again \cite{bombierizanniernorthcottproperty}, we say that  a set of algebraic numbers \( S \) has the \emph{Bogomolov property (B)} if there exists a constant \( c=c(S)>0 \) such that for every \( \alpha \in S \), either \( h(\alpha)=0 \) or \( h(\alpha) \geq c \).

In \cite{Amoroso_David_Zannier-OnFieldWithPropertyB}, the authors studied property (B), for infinite Galois extensions of number fields, also in relation to the Galois group of the extension. They also introduced property (B) for profinite groups \cite[Definition 5.5]{Amoroso_David_Zannier-OnFieldWithPropertyB}. Here, we present an extension of their definition, in order to include, for instance, properties (N) and ($\mathfrak{S}$).
Let $(\star)$ be some property of a subfield of $\overline{\mathbb{Q}}$. 
\begin{dfn}\label{DefinitionProperty(S)}
    A profinite group $G$ is said to have property $(\star)$ over a number field $K$ if, for every Galois extension $L/K$ such that $\mathrm{Gal}(L/K)$ is isomorphic to $G$, $L$ has property $(\star)$.
\end{dfn}
By Northcott's theorem, every finite group satisfies properties (N) and (B) over any number field. By \cite[Proposition 2.1]{abelianANDexponentbounded}, any abelian group with finite exponent has property (N) over any number field. Furthermore, by \cite[Corollary 1.7]{Amoroso_David_Zannier-OnFieldWithPropertyB} and \cite[Theorem 1.2]{theseSara}, any group such that \( G/Z(G) \) has finite exponent satisfies property (B) over any number field. Additionally, as noted in \cite[Remark p.8]{bombierizanniernorthcottproperty}, every finite group satisfies property \((\mathfrak{S})\) over \(\Q\). The second goal of this article is to study property \((\mathfrak{S})\) over \(\Q\) for infinite profinite groups.
 
\begin{Question}\label{MainQuestion}
Does there exist an infinite profinite group $G$ that satisfies property $(\mathfrak{S})$ over $\Q$ ?
\end{Question}
In this article, we provide some evidence supporting a negative answer to this question. 
A result by Checcoli and Fehm \cite[Theorem 1.2]{checcolinorthcottproperty} states that any family of finite solvable groups \( (G_i)_{i \in \mathbb{N}} \) can be realized over \( \mathbb{Q} \) by an extension with property \( (\mathfrak{S}) \). This suggests that direct products of solvable groups are particularly interesting candidates to examine.

In \cite[Theorem 1]{CheccoliManuscripta} Checcoli showed that certain direct products of solvable groups with unbounded exponent can be realized with infinite local degrees at each prime, and hence do not satisfy property \((\mathfrak{S})\) since \(\mathfrak{S}(L)\) is trivially zero. 

The second result of this article provides examples of groups that, although we know some of them can be realized over \( \mathbb{Q} \) by an extension with property \( (\mathfrak{S}) \), do not themselves satisfy property \( (\mathfrak{S}) \).
Moreover, this occurs for nontrivial reasons, as \cite[Theorem 1]{CheccoliManuscripta} does not apply in the setting of groups without any restriction on the exponent.

\begin{thm}\label{Theorem Negative answer to 2nd problem}
    Let ${G} = \prod_{i=1}^{\infty} G_{i}$ be a direct product of nontrivial finite solvable groups or symmetric groups. Suppose that, for infinitely many  indices $i$, $G_i$ belongs to one of the following families:
    \begin{enumerate}[(i)]
       \item \label{Groups of odd order} Groups of odd order ;
        \item \label{Abelian groups} Abelian groups ;
        \item \label{Hamiltonian groups} Hamiltonian groups ;
        \item \label{Symmetric groups} Symmetric groups ;
        \item \label{Groups of even order having a cyclic 2-Sylow subgroup} Groups of even order having a cyclic $2$-Sylow subgroup. 
    \end{enumerate} 
    Then $G$ does not have property $(\mathfrak{S})$ over $\Q$.
\end{thm} 

The key point in the proof of Theorem \ref{Theorem Negative answer to 2nd problem} is a criterion, referred to as property \( (\mathcal{DL}) \) (see Definition \ref{def-spade-suit}), for a finite group to admit sufficiently many linearly disjoint realizations over the rationals with certain local conditions. We investigate how property \( (\mathcal{DL}) \) behaves under direct products of groups in Section \ref{Subsection General Criteria for a Direct Product Not to Have Property S}. 
We then demonstrate that each class of groups mentioned in the statement of Theorem \ref{Theorem Negative answer to 2nd problem} satisfies this criterion. The proof relies, among other tools, on basic results on the resolution of certain embedding problems (Section \ref{SectionEmbeddingProblem}), the Grunwald problem (Subsection \ref{Subsection Grunwald solvable extension}), and auxiliary results on the splitting of primes (Proposition \ref{givenPrimesRemainsInQ(sqrt(q))}). 
In the final section, we further explore property ($\mathfrak{S}$) for fields. Specifically, we compare it to the result of Bombieri and Zannier \cite[Theorem 2]{bombierizanniernorthcottproperty} and the criterion of Widmer \cite[Theorem 3]{widmernorthcottcriteria}, showing the existence of infinite extensions of $\mathbb{Q}$ satisfying all three properties. Lastly, we use a result of Bombieri and Zannier \cite[Theorem 2.1 and Example 2]{bombierizanniernorthcottproperty} to derive a weak version of a result from \cite[Theorem 3]{NorthcottNumberPazuki} on the Northcott number of a field.

\section{Some preliminaries on splitting of primes in Galois extensions}\label{prel-splitting}
\subsection{On totally split primes in finite Galois extensions of $\Q$} \label{Subsection Auxiliary result density prime}

In this subsection, we denote by $\mathcal{T}(F)$ the set of rational prime numbers that split totally in a given number field $F$.

We also use the notation $f(x) \underset{x\rightarrow + \infty}{\sim} g(x)$ if,  for large $x$, 
$$
f(x)=h(x)g(x) \quad \text{with} \quad \underset{x \rightarrow +\infty}{\lim} h(x)=1.
$$  

We begin by proving that every finite group $G$ has property $(\Sgothique)$.
This was already proved for instance in \cite[Proposition 2.1]{checcolinorthcottproperty}. We give a slightly different proof, also based on Chebotarev's density theorem and the prime number theorem. The advantage of this proof is that we can easily use it to also prove Proposition \ref{cas p congru à 1 modulo 4}.
\begin{prop} \label{numbers field case}
Let $K/\mathbb{Q}$ be a finite Galois extension. Then 
\[
\sum_{p \in \mathcal{T}(K)} \frac{1}{p}=+\infty.\] In particular $\Sgothique(K)=+\infty$.
\end{prop}

\begin{proof}
For $x>0$, we set 
$$
\pi_{K}(x)=|[0,x] \cap \mathcal{T}(K)|.$$
Since $K$ is Galois, by Chebotarev's density theorem (\cite[Theorem 4]{ChebotarevDensityTheorem}, \cite{ThesisChebotarev}), the natural density of $\mathcal{T}(K)$ is $1/[K:\mathbb{Q}]$, that is
$$
\pi_{K}(x) \underset{x\rightarrow + \infty}{\sim} \frac{x}{[K:\mathbb{Q}]\log(x)}.
$$ 
So we have  
\begin{align*} 
\pi_{K}(x) \log \pi_{K}(x) &  \underset{x\rightarrow + \infty}{\sim}  \frac{x}{[K:\mathbb{Q}]\log(x)} \log \left( \frac{x}{[K:\mathbb{Q}]\log(x)} \right) 
\underset{x\rightarrow + \infty}{\sim} \frac{x}{[K:\mathbb{Q}]},
\end{align*}
where, for the first equivalence, we used the elementary fact that 
$$
\log(f(x)) \underset{x\rightarrow + \infty}{\sim} \log(g(x)) \quad \text{if} \quad f(x) \underset{x\rightarrow + \infty}{\sim} g(x) \quad \text{and} \quad \underset{x \rightarrow + \infty}{\lim} g(x)=+\infty.
$$

Therefore, if we denote by $p_{n}$ the $n$-th element of $\mathcal{T}(K)$, by noting that $\pi_{K}(p_{n})=n$ and applying the above result with $x=p_{n}$, we have 
$$
p_{n} \underset{n\rightarrow + \infty}{\sim} [K:\mathbb{Q}]n \log n.
$$ 
Since 
$$
\sum_{n=2}^{+\infty} \frac{1}{n \log n}=+\infty, 
$$ 
we deduce the sought-for result.
\end{proof}

Before proceeding, we recall the following well-known result, whose proof can be found in~\cite[Lemma 3.3.30]{HenriCohen}.

\begin{lemma} \label{Lemme split totally in compositum}
    Let $K$ be a number field, $\mathfrak{p}$ a prime ideal of $K$, and $L, L'$ two finite extensions of $K$. Then $\mathfrak{p}$ splits completely in $L/K$ and $L'/K$ if and only if it splits completely in the compositum extension $LL'/K$.
\end{lemma}

Now we are going to examine which contribution to the divergence of the sum in \ref{numbers field case} come from the primes that are congruent to $1$ or to $3$ modulo $4$. 

\begin{prop} \label{cas p congru à 1 modulo 4}
    Let $K/\mathbb{Q}$ be a finite Galois extension. 
\begin{enumerate}[(i)]
\item\label{item-cas p congru à 1 modulo 4}
The set $\{p\in \mathcal{T}(K)\mid p \equiv 1 \bmod{4} \}$ is infinite and 
    $$
    \sum_{\substack{p \equiv 1 \bmod{4} \\ p \in \mathcal{T}(K)}} \frac{1}{p} = +\infty.
    $$
\item\label{item-cas p congru à 3 modulo 4}
  The set  $
    \{p \in  \mathcal{T}(K) \mid p \equiv 3 \bmod{4}\}$ is non-empty   if and only if $ i \notin K$.
 Moreover, in this case the set is infinite, and 
    $$
    \sum_{\substack{p \equiv 3 \bmod{4} \\ p \in \mathcal{T}(K)}} \frac{1}{p} = +\infty.
    $$
   % In particular, we have 
   % $$
    %\sum_{\substack{p \equiv 3 \bmod{4} \\ p \in \mathcal{P}}} \Sgothique_{p}(K) = +\infty.
    %$$
\end{enumerate}
\end{prop}

\begin{proof}
   To prove part \eqref{item-cas p congru à 1 modulo 4}, notice that $K(i)$ is Galois over $\mathbb{Q}$ as the compositum of two Galois extensions. By Lemma \ref{Lemme split totally in compositum} , we have 
   $$
   \mathcal{T}(K(i))= \mathcal{T}(\mathbb{Q}(i)) \cap \mathcal{T}(K) = \{p \in \mathcal{P} \mid p \equiv 1 \bmod{4}\} \cap \mathcal{T}(K)
   $$
which is an infinite set.
   We then conclude by applying Proposition \ref{numbers field case} to $K(i)/\mathbb{Q}$.

To prove \eqref{item-cas p congru à 3 modulo 4}, notice that $
    \mathcal{T}(\mathbb{Q}(i)) = \{p \in \mathcal{P} \mid p \equiv 1 \bmod{4} \},
    $
    so if $i \in K$, then $\{p \in \mathcal{T}(K) \mid p \equiv 3 \bmod{4}\} = \emptyset$.

    Now let's suppose that $i \notin K$, and let $L = \mathbb{Q}(i)$. We have 
\[
        \{ p \in \mathcal{T}(K) \mid p \equiv 2,3 \bmod{4} \} = \mathcal{T}(K) \setminus (\mathcal{T}(L) \cap \mathcal{T}(K)).
\]  Since $i\notin K$ we get $[KL: \mathbb{Q}] = 2[K:\mathbb{Q}]$. Since $KL$ and $K$ are Galois over $\mathbb{Q}$ we conclude from Chebotarev's density theorem that the natural density of the sets $\mathcal{T}(K)$ and $\mathcal{T}(KL)$ are $1/[K:\mathbb{Q}]$ and $1/[KL:\mathbb{Q}]$ respectively. From Lemma \ref{Lemme split totally in compositum} we get $\mathcal{T}(L) \cap \mathcal{T}(K)=\mathcal{T}(KL)$. Therefore, the natural density of the set  $\mathcal{T}(K) \setminus (\mathcal{T}(L) \cap \mathcal{T}(K))$ is given by
    $$
    \frac{1}{[K:\mathbb{Q}]} - \frac{1}{[KL:\mathbb{Q}]} = \frac{1}{2[K:\mathbb{Q}]},
    $$
    which means that 
    $$
    | \{p \le x \mid p \in \mathcal{T}(K) \text{ and } p \equiv 3 \pmod{4} \} | \underset{x \rightarrow +\infty}{\sim} \frac{1}{2[K:\mathbb{Q}]} \frac{x}{\log(x)}.
    $$
    And we conclude as in the proof of Proposition \ref{numbers field case}, by noting that the $n$-th prime number that is congruent to $3 \pmod{4}$ and that splits totally in $K$, is equivalent to $2[K:\mathbb{Q}] n \log n$ as $n$ approaches infinity.
\end{proof}

%\begin{rmq}
%    For any given number field $K$, not necessarily Galois, we can generalize Propositions \ref{numbers field case}, \ref{cas p congru à 1 modulo 4}, and \ref{cas p congru à 3 modulo 4} by taking the Galois closure $M$ of $K$ and using the fact that the primes that split completely in $M$ are the same that split completely in $K$ (see \cite[Exercise 4 p.53]{splitiifsplitGaloisclosure}). The natural density will therefore be $\frac{1}{[M:\mathbb{Q}]}$ or $\frac{1}{2[M:\mathbb{Q}]}$. 
  %  We just need to be careful about whether $i \in M$ or not for the generalization of Proposition \ref{cas p congru à 3 modulo 4}. Indeed, we can have a number field $K$ for which $i \notin K$, but $i \in M$. An example of such a field can be constructed, for instance, using \cite[Theorem 1.1]{givenquadraticextensionfindcubicextension}, by taking $M$ as a $\Sgothique_{3}$-extension that contains the quadratic extension $\mathbb{Q}(i)$. We take $K$ as one of the three sub-cubic fields of $M$. Hence, $K$ fits the claims. We can also simply take $K = \mathbb{Q}(\sqrt[4]{2})$.
%\end{rmq}

\subsection{Quadratic extensions with given splitting behaviour} \label{Subsection Grunwald Problem}
In this subsection, given a finite set of odd primes, we aim to construct quadratic number fields which exhibits a prescribed splitting behavior at the given primes.  
\begin{prop} \label{nombre premiers donnés qui split et ou qui reste inert}
   Let $\mathcal{S}$, $\mathcal{I}$, and $\mathcal{R}$ be three finite and pairwise  disjoint sets of odd prime numbers. Then there exists a quadratic totally real extension $F/\Q$ and  such that:
\begin{enumerate}[(i)]
\item\label{item-split} all primes of $\mathcal{S}$ split totally in $F$;
\item\label{item-inert} all primes of $\mathcal{I}$ remain inert in $F$;
\item\label{item-ram} all primes of $\mathcal{R}$ ramify in $F$. 
\end{enumerate}
\end{prop}
\begin{proof}
For a finite set of primes $\mathcal{M}$, let $P_{\mathcal{M}}$ denote the product of all its elements. 
% Let $$S = \{p_1, \ldots, p_n\}, I = \{p'_1, \ldots, p'_r\}, \text{ and } R = \{\tilde{p}_1, \ldots, \tilde{p}_l\}.$$ 
By the Chinese remainder theorem, we have a ring isomorphism
   $$\varphi:  \prod_{p\in \mathcal{S}} \mathbb{Z}/p\mathbb{Z} \times \prod_{p'\in \mathcal{I}} \mathbb{Z}/p'\mathbb{Z} \times \prod_{\ell\in \mathcal{R}} \mathbb{Z}/\ell^2\mathbb{Z} \rightarrow \mathbb{Z}/ \left(P_{\mathcal{S}}P_{\mathcal{I}}P_{\mathcal{R}}^2 \right) \mathbb{Z}.$$ 
   
   Let $a_1, \ldots, a_{|\mathcal{I}|}$ be integers such that,  for $1 \le j \le |\mathcal{I}|$, $a_j$ is not a quadratic residue modulo any prime in $\mathcal{I}$. Set $\mathcal{R}=\{\ell_1, \ldots,\ell_{|\mathcal{R}|}\}$ and let $m\in \mathbb{N}\setminus\{0\}$ be any representative of the class
   $$\varphi(1, \ldots, 1, a_1, \ldots, a_{|\mathcal{I}|}, \ell_1, \ldots,\ell_{|\mathcal{R}|})  \in  \mathbb{Z}/ \left(P_{\mathcal{S}}P_{\mathcal{I}}P_{\mathcal{R}}^2 \right) \mathbb{Z}.$$ 
   
 %  Here, we can choose $m > 0$ or $m < 0$ to have a totally real or totally complex extension, respectively. 

Let
   $$m= \prod_{i=1}^{s} q_i^{2m_i + r_i}$$ be the factorization of $m$ into distinct primes $q_1,\ldots,q_s$, where $m_i\in \mathbb{N}$ and  $r_i \in \{0,1\}$ and set 
  \begin{equation} \label{definition of m'}
      m' = m \prod_{i=1}^{s} q_i^{-2m_i} = \prod_{i=1}^{s} q_i^{r_i} \in \mathbb{N}\setminus\{0\}.
  \end{equation}

We claim that the field $F=\Q(\sqrt{m'})$ satisfies the conditions in the statement. Without loss of generality, we may assume that $m' \neq 1$, i.e., that $m$ is not a square. Indeed, it suffices to take an additional odd prime $q$ in the application of the Chinese Remainder Theorem and require that the class of $m$ modulo $q$ is not a quadratic residue.
 Clearly, as $m'\in \mathbb{N}_{\ge 2}$ is square-free, $F/\Q$ is a quadratic totally real extension.

Notice that the discriminant of $F$ is either $m'$ or $4m'$. Since by the hypothesis and the definition of $m$, we have 
   $\{2,q_1, \ldots, q_s\} \cap \left(\mathcal{S} \cup \mathcal{I}\right) = \emptyset $, every prime in $\mathcal{S} \cup \mathcal{I}$ does not ramify in $F$. 

Let $p\in\mathcal{S}$. As $m\equiv 1 \pmod p$ it follows that each $q_{i}$ has an inverse modulo $p$. Hence, we infer from (\ref{definition of m'}) that 
   $$m' \equiv \left(\prod_{i=1}^{s} q_i^{-m_i}\right)^2 \pmod{p}$$ 
   and this shows that $m'$ is a quadratic residue modulo $p$, so $p$ splits totally in $F$, proving \eqref{item-split}.

 To prove \eqref{item-inert}, notice also that, for every $\ell\in \mathcal{R}$, $m \equiv \ell \pmod{\ell^2}$, hence $\ell$ divides exactly $m$ and thus, $\ell$ divides $m'$. For every prime $p'\in\mathcal{I}$, there exists $1\leq j\leq |\mathcal{I}|$ such that
   $$m' \equiv a_j \prod_{i=1}^{s} q_i^{-2m_i} \pmod{p'}$$ 
   so $m'$ is not a quadratic residue modulo $p'$, otherwise $a_j$ would be a quadratic residue modulo some $p'\in \mathcal{I}$, contradicting the hypothesis. Therefore, all elements in $\mathcal{I}$ remain inert in $F$.

Finally, by the definition of $m$, we also have $\mathcal{R} \subset \{q_1, \ldots, q_s\}$, therefore every prime of $\mathcal{R}$ ramifies in $F$, proving \eqref{item-ram}.
\end{proof}
\begin{rmq}
We assumed that the sets in Proposition \ref{nombre premiers donnés qui split et ou qui reste inert} only contain odd primes for simplicity, as this suffices for our application, but we can easily also control the splitting behavior at the prime $2$. Indeed, it is enough to add the extra condition in the application of the Chinese Remainder Theorem that $m \equiv 2 \pmod{4}$ if we want 2 to ramify,   $m \equiv 1 \pmod{8}$ if we want $2$ to split completely,  or $m \equiv 5 \pmod{8}$ if we want $2$ to remain inert. We can also require the extension \(F/\mathbb{Q}\) to be totally complex by choosing a negative integer \(m < 0\) as the representative of the class in the proof.
 
\end{rmq}
\begin{rmq}
We notice that a non-explicit proof of the above result follows directly from the Grunwald–Wang Theorem \cite[Main Theorem]{WangGrunwaldAbelianCase} (which answers a special case of the Grunwald Problem). 
This asserts that, if $K$ is a number field, $G$ is a finite abelian group, $\mathcal{M}$ is a finite set of places of $K$ not containing $2$, and, for $v\in \mathcal{M}$, $L_v/K_v$ is a finite Galois extension with Galois group $G_v$ embeddable in $G$, then there exists a Galois extension of number fields $L/K$ with Galois group $G$ such that the completion of $L$ at any place lying above $v$ is exactly $L_v$.
Then Proposition \ref{nombre premiers donnés qui split et ou qui reste inert} follows by taking $\mathcal{M}=\mathcal{S}\cup\mathcal{I}\cup\mathcal{R}$ and  $L_p=\Q_p$ if $p \in \mathcal{S}$,  $L_p$ equal to the unique unramified quadratic extension of $\mathbb{Q}_p$ if $p \in \mathcal{I}$, or a quadratic ramified extension of $\mathbb{Q}_p$ if $p \in \mathcal{R}$ and either a trivial or a non-trivial extension for the place at infinity, depending on whether we want the extension to be totally real or not.
\end{rmq}

\section{Proof of Theorem \ref{Theorem Property S does't preserved under finite extension}} \label{answer to question about preservation under finite extension}
Before proceeding with the proof, we give an elementary result on the behavior of $\Sgothique$ on composita. Given a set $I\subset \mathbb{N}$ and a set $(L_i)_{i\in I}$ of (possibly infinite) extensions of $\Q$, we denote by $\prod_{i\in I} L_i$ their compositum in $\overline{\Q}$.
We have the following:
\begin{lemma} \label{lemme elementaire}
    Let $I\subset \mathbb{N}$ and let $(L_{i})_{i \in I}$ be a set of Galois extensions over $\mathbb{Q}$. Then we have 
    $$\mathfrak{S}\left(\prod_{i \in I} L_{i}\right) \le \sum_{p \in \mathcal{P}} \min_{i \in I} \left(\mathfrak{S}_{p}\left(L_{i}\right)\right).$$
\end{lemma}

\begin{proof} Set $L=\prod_{i \in I} L_{i}$. 
    Since $L/\mathbb{Q}$ is Galois as the compositum of Galois extensions, $\mathfrak{S}(L)$ is well defined. We have 
 \[ \mathfrak{S}(L) =
        \sum_{p \in \mathcal{P}} \frac{\log p}{e_{p}(L)\left(p^{f_{p}(L)} + 1\right)}.\]
Now the lemma follows noticing that
    $e_{p}(L) \ge \max_{i \in I} \left(e_{p}(L_{i})\right)$ and $f_{p}(L) \ge \max_{i \in I} \left(f_{p}(L_{i})\right)$.
\end{proof}
We are now able to achieve our goal, namely to construct  an infinite Galois extension $L/\mathbb{Q}$ that satisfies $\mathfrak{S}(L)=+\infty$, but for which there exists a finite extension $F/L$ such that $F/\mathbb{Q}$ is Galois and $\mathfrak{S}(F)<+\infty$. 

\begin{proof}[Proof of Theorem \ref{Theorem Property S does't preserved under finite extension}]
% PUT THIS in the INTRO
%To this end, we will construct an infinite Galois extension $L/\mathbb{Q}$ such that
%$$
%\sum_{\substack{p \equiv 1 \bmod{4} \\ p \in \textbf{S}(L)}} \frac{\log p}{e_{p}(L)(p^{f_{p}(L)}+1)} < +\infty \text{ and } \sum_{\substack{p \equiv 3 \bmod{4} \\ p \in \textbf{S}(L)}} \frac{\log p}{e_{p}(L)(p^{f_{p}(L)}+1)} = +\infty,
%$$
%which cannot occur in the finite case according to Proposition \ref{cas p congru à 1 modulo 4}. In particular we will have $\mathfrak{S}(L)=+\infty$, and we can check that in this case we have $\mathfrak{S}(L(i))<+\infty$.
 Suppose that there exist an increasing sequence of positive integers $(n_k)_{k \geq 1}$  and a sequence of quadratic totally real number fields $(F_k)_{k \geq 1}$ such that,  for every $k \geq 1$, the compositum $L_k=F_1 \dots F_k$ satisfies:
\begin{enumerate}[(a)]
    \item\label{(c)} $F_{k+1}$ and $L_{k}$ are linearly disjoint over $\mathbb{Q}$.
    \item\label{(b)} Every prime $p \leq n_k$ that is congruent to $3 \pmod{4}$ splits totally in $F_k$, and every prime $p \leq n_k$ that is congruent to $1 \pmod{4}$ remains inert in $F_k$.
\item\label{(a)} $
    \sum_{\substack{p \in \mathcal{P}_{k+1}}} \mathfrak{S}_p(L_{k}) \geq 1
    $, where $\mathcal{P}_{k+1}=\{p\in\mathcal{P}\mid  p \equiv 3 \bmod{4},  n_{k} \leq p < n_{k+1}\}$.
\end{enumerate}
Then we claim that the fields $L = \prod_{k \geq 1} F_k$ and $F=L(i)$ satisfy the statement
of Theorem \ref{Theorem Property S does't preserved under finite extension}.

Indeed,  by condition \eqref{(c)}, $L/\mathbb{Q}$ is Galois with Galois group the direct product $\prod_{k \geq 1} \mathbb{Z}/2\mathbb{Z}$ and, as all the fields $F_i$ are totally real, $i\notin L$. Hence $F/L$ is a quadratic extension and 
$F/\mathbb{Q}$ is Galois as it is the compositum of two Galois extensions.

Furthermore, by condition (\ref{(b)}), it holds for every \(k \in \mathbb{N}\backslash \{0\}\) and every \(p \in \mathcal{P}_{k}\) that
\(\mathfrak{S}_{p}(L) = \mathfrak{S}_{p}(L_{k-1})\). Indeed, for each \(p \in \mathcal{P}_{k}\), the prime \(p\) splits completely in every fields \(F_{l}\) with \(l \geq k\), which implies that the ramification index and the inertia degree satisfy \(e_{p}(L) = e_{p}(L_{k-1})\) and \(f_{p}(L) = f_{p}(L_{k-1})\), respectively.

Thus we have 
\begin{align*}
    \mathfrak{S}(L) & = \sum_{p \in \mathcal{P}} \mathfrak{S}_p(L) 
     \geq \sum_{\substack{p \in \mathcal{P} \\ p \equiv 3 \bmod{4}}} \mathfrak{S}_p(L) 
     = \sum_{k=1}^{+\infty} \sum_{\substack{p \in \mathcal{P}_k }} \mathfrak{S}_p(L) \\
    & = \sum_{k=1}^{+\infty} \sum_{\substack{p \in \mathcal{P}_k }} \mathfrak{S}_p(L_{k-1}) 
     \overset{\eqref{(a)}}{\geq} \sum_{k=1}^{+\infty} 1 = \infty.
\end{align*}

On the other hand,  using also Lemma \ref{lemme elementaire}, we have that
\begin{align*}
    \mathfrak{S}(F) &= \mathfrak{S}(L\mathbb{Q}(i)) 
    = \mathfrak{S} \left( \mathbb{Q}(i) \prod_{k \geq 1} F_k \right) 
    \\
    &\leq \sum_{p \in \mathcal{P}} \min \left( \min_{k \in \mathbb{N}} \mathfrak{S}_p(F_k), \mathfrak{S}_p(\mathbb{Q}(i)) \right)\\
    & \leq \sum_{\substack{p \in \mathcal{P} \\ p \equiv 1 \bmod{4}}} \min_{k \in \mathbb{N}} \mathfrak{S}_p(F_k) + \sum_{\substack{p \in \mathcal{P} \\ p \not\equiv 1 \bmod{4}}} \mathfrak{S}_p(\mathbb{Q}(i)) \\
    & \overset{\eqref{(b)}}\leq \sum_{\substack{p \in \mathcal{P} \\ p \equiv 1 \bmod{4}}} \frac{\log(p)}{e_p(L)(p^2 + 1)} + \sum_{\substack{p \in \mathcal{P} \\ p \not\equiv 1 \bmod{4}}} \frac{\log(p)}{p^2 + 1}  
     \leq \sum_{p \in \mathcal{P}} \frac{\log(p)}{p^2} 
    < +\infty.
\end{align*}
This  proves our claim.

Now, to prove Theorem \ref{Theorem Property S does't preserved under finite extension}, we are left to show the existence of two sequences $(n_k)_{k\geq 1}$ and $(F_k)_{k\geq 1}$ satisfying conditions  \eqref{(c)}, \eqref{(b)} and \eqref{(a)}.

We construct such sequences inductively.
Set $n_0=1$ and $F_0=\mathbb{Q}$.
Suppose we have already constructed $n_1, \dots, n_{k}$ and $F_1, \dots, F_{k}$. Since the fields $F_{i}$ are all quadratic totally real fields, their compositum $L_{k}$ is also totally real, so $i \notin L_{k}$. Moreover  $L_{k}/\Q$ is Galois as the compositum of Galois extensions. Thus, by Proposition  \ref{cas p congru à 1 modulo 4}\eqref{item-cas p congru à 3 modulo 4}, we have 
\[
\sum_{\substack{p \in \mathcal{P} \\ p \equiv 3 \bmod{4}}} \mathfrak{S}_p(L_{k}) = +\infty.
\]
Thus, there exists $n_{k+1}>n_{k}$ such that condition \eqref{(a)} holds.

According to Proposition \ref{cas p congru à 1 modulo 4}\eqref{item-cas p congru à 1 modulo 4}, we can take a rational prime $q$ that is congruent to $1 \pmod{4}$ and splits totally in $L_{k}$. Now, by Proposition \ref{nombre premiers donnés qui split et ou qui reste inert}, we can take a quadratic real field $F_{k+1}$ such that all rational prime numbers $p \leq n_{k+1}$ congruent to $3 \pmod{4}$ split totally in $F_{k+1}$, and all rational prime numbers $p \leq n_{k+1}$ congruent to $1 \pmod{4}$ and $q$ remain inert in $F_{k+1}$.

Then $F_{k+1}$ satisfies \eqref{(b)} by definition. Let's suppose, by contradiction, that $F_{k+1}/\mathbb{Q}$ is not linearly disjoint from $L_{k}/\mathbb{Q}$. Since $F_{k+1}/\mathbb{Q}$ is a quadratic extension, we necessarily have $F_{k+1} \subseteq L_{k}$. However, $q$ is a rational prime number that remains inert in $F_{k+1}$ and splits totally in $L_{k}$, which is a contradiction. Therefore, condition \eqref{(c)} holds.
\end{proof}

%In fact, we have precisely
 %   \[
%    \sum_{\substack{p \in \mathcal{P} \\ p \equiv 1 \bmod{4}}} \mathfrak{S}_p(L) = \sum_{\substack{p \in \mathcal{P} \\ p \equiv 1 \bmod{4}}} \mathfrak{S}_p(F) = \sum_{\substack{p \in \mathcal{P} \\ p \equiv 1 \bmod{4}}} \frac{\log(p)}{e_p(L)(p^2 + 1)}
  %  \]
%    because for every rational prime number $p \leq n_k$ congruent to $1 \pmod{4}$, the $p$-adic completion of $F_k$ and $F_{k+1}$ are the same since $\mathbb{Q}_p$ has a unique unramified quadratic extension.

% Using a similar construction and the ramification part of Proposition \ref{nombre premiers donnés qui split et ou qui reste inert}, we can create an infinite Galois extension $L/\mathbb{Q}$ such that $\mathfrak{S}(L) = +\infty$ and $\mathfrak{T}(L) = \sum_{p \in \mathcal{T}(L)} \frac{1}{p} < \infty$ (in fact, we can even have $\mathcal{T}(L) = \emptyset$), which again does not occur in the finite case according to Proposition \ref{numbers field case}.

\begin{rmq} \label{Remark Referee Question}
It may be possible to combine the construction of Checcoli and Fehm in \cite[Section~3]{checcoli2024widmerscriterianorthcottproperty}, the ideas developed in this paper, and certain known cases of the Grunwald problem (such as \cite[Main Theorem]{WangGrunwaldAbelianCase}, \cite[Corollary~2]{GrunwaldSolvableCase}, or \cite[Theorem~5.1 and Theorem~5.9]{GenericGaloisExtension}) to construct a field with property~$(N)$ for which all known criteria fail, yet which contains a finite-index subfield with property~$(\mathfrak{S})$. This could provide a natural application of property~$(\mathfrak{S})$ and may be worth exploring in future work.
\end{rmq}

%Additionally, this shows that a countable direct product of solvable groups with infinitely many copies of $\mathbb{Z}/2\mathbb{Z}$, or $\mathbb{Z}/q\mathbb{Z}$ with $q$ an odd prime if we modify our construction with the Grunwald-Wang theorem, does not have property $(\mathfrak{S})$.

\section{Property (B), (N) and $(\mathfrak{S})$ for groups}
The goal of this section is to study property ($\mathfrak{S}$) for groups and motivate Question \ref{MainQuestion} and Theorem \ref{Theorem Negative answer to 2nd problem}. Referring back to Definition \ref{DefinitionProperty(S)}, we have the following:

\begin{prop}\label{PropositionExampleGroupHavingProperty(N),(B)} Let $G$ be a profinite group. Then:
\begin{enumerate}[(i)]
\item\label{it-prop-i} $G$ has property $(\mathfrak{S})$ over $\Q$ $\Rightarrow$ $G$ has property (N) over $\Q$ $\Rightarrow$ $G$ has property (B) over $\Q$ and, in general, none of the converse implications holds.
\item\label{it-prop-ii} If $G$ is finite, $G$ has properties ($\mathfrak{S}$) over $\mathbb{Q}$ and properties (N) and (B) over any number field.
\item\label{it-prop-iii} If $Z(G)$ denotes the center of $G$ and if the quotient $G/Z(G)$ has finite exponent, then $G$ has property (B) over any number field.
\item\label{it-prop-iv} If $G$ is abelian with finite exponent, then $G$ has property (N) over any number field.
\end{enumerate}
\end{prop}

\begin{proof} For point \eqref{it-prop-i}, the first implication is a direct consequence of \cite[Theorem 2 and Remark p.8]{bombierizanniernorthcottproperty}, and the second implication follows easily from the Northcott property. For the converse,  the profinite group \( G = \mathrm{Gal}(\mathbb{Q}^{\mathrm{ab}}/\mathbb{Q}) \) has property \((B)\) by \cite[Corollary 1.7]{Amoroso_David_Zannier-OnFieldWithPropertyB} but does not have property \((N)\) since $\mathbb{Q}^{\mathrm{ab}}$ contains all the roots of unity. The group \( G = \prod_{n \in \mathbb{N}} \mathbb{Z}/2\mathbb{Z} \) has property (N) by \cite[Theorem 1]{bombierizanniernorthcottproperty}, but not property $(\Sgothique)$ as shown in the construction in Section \ref{answer to question about preservation under finite extension}.

Point \eqref{it-prop-ii} follows from Northcott's theorem for properties (N) and (B), while for property $(\mathfrak{S})$ it follows from  \cite[Proposition 2.1]{checcolinorthcottproperty} or Proposition \ref{numbers field case}. 

 Point \eqref{it-prop-iii} is a direct consequence of \cite[Corollary 1.7]{Amoroso_David_Zannier-OnFieldWithPropertyB}. 
 
 As for \eqref{it-prop-iv}, this essentially follows from \cite[Proposition 2.1]{abelianANDexponentbounded}. Since the result therein was proved only for abelian extension of $\Q$, we recall the proof here.    Let \( G \) be an abelian group such that \(\exp(G) \leq T < +\infty\). Consider a Galois extension \( L/K \) where \( K \) is a number field and \(\mathrm{Gal}(L/K)\) is isomorphic to \( G \). Let \( K \subset F \subset L \) be such that \( F/K \) is a finite Galois extension. By the structure theorem for finite abelian groups, \(\mathrm{Gal}(F/K)\) is a direct product of cyclic groups of order at most \( T \). Therefore, we have \( F \subset K_{\mathrm{ab}}^{(T)} \). Since \( L \) is the compositum of all finite Galois extensions over \( K \) contained in \( L \), we deduce that \( L \subset K_{\mathrm{ab}}^{(T)} \). Now, by \cite[Theorem 1]{bombierizanniernorthcottproperty}, we conclude that \( L \) has property \((N)\) since \( K_{\mathrm{ab}}^{(T)} \) does.
\end{proof} We also remark that the group $\mathrm{Gal}(\Q^{tr}/\Q)$ has property (B) over $\Q$, but not over $\Q(i)$ (see \cite[paragraph after Theorem 5.4]{Amoroso_David_Zannier-OnFieldWithPropertyB}).

A natural question is whether profinite groups with finite exponent have property (N) over any number field.
This problem is open, even for the group \( G = \mathrm{Gal}(\mathbb{Q}^{(3)}/\mathbb{Q}) \) and its realization \(\mathbb{Q}^{(3)}/\Q\).

\section{A general criterion for the failure of Property $(\mathfrak{S})$ for direct products of groups} \label{Subsection General Criteria for a Direct Product Not to Have Property S}
The following group property allows the realization of the group over 
 $\mathbb{Q}$ in a flexible way, combining a linear disjointness condition with a local constraint.
\begin{dfn}[Disjointness-Local Property]\label{def-spade-suit} A finite group $G$ is said to have \emph{property ($\mathcal{DL}$)} if for every finite Galois extension $L/\Q$, there is a Galois extension  $K/\Q$ such that $\mathrm{Gal}(K/\Q) = G$, $K\cap L = \Q$ and $f_{p_L}(K)\geq 2$, where $p_L=\min\{p \in\mathcal{P}\setminus\{2\}\mid f_{p}(L)=1\}$. 
\end{dfn}

The following diagram provides a schematic representation of this technical definition. 
\[
\begin{tikzpicture}[node distance=1.5cm, auto]
  \node (k)      at (0.25,1)     {\(  K \cap L =\Q \)};
  \node (L)      at (-2,3)      {\( L \mbox{  (arbitrary finite Galois extension)} \)};
  \node (Qz)     at (3.5,3)       {\( K \mbox{ (Galois over }\Q\mbox{)}  \)};

  \draw[-] (k) -- (L) node[midway, left]{$f_{p_{L}}(L)=1$\;};
  \draw[-] (k) -- (Qz) node[midway, right] {\;$\mathrm{Gal}(K/\Q)=G$ and $f_{p_{L}}(K)>1$};
\end{tikzpicture}
\]

It is worth noting that the group of two elements satisfies property \((\mathcal{DL})\), as can be seen from Proposition \ref{nombre premiers donnés qui split et ou qui reste inert}, whereas the trivial group does not satisfy property \((\mathcal{DL})\).

We have the following lemma:
\begin{lemma}\label{dir-prod-spade}
Let $G$ be a finite group with property ($\mathcal{DL}$) and let $G'$ be a finite solvable group. Then the product $G\times G'$ has also property ($\mathcal{DL}$). 
\end{lemma}
\begin{proof}
Let $L/\Q$ be a finite Galois extension and let $p_L$ be the smallest odd prime such that $f_{p_L}(L)=1$. By hypothesis there is a Galois extension  $K/\Q$ such that $\mathrm{Gal}(K/\Q) = G$, $K\cap L = \Q$ and $f_{p_L}(K)\geq 2$. Since, for every \(n \geq 1\), the group \(G'^n\) is solvable, it follows from Shafarevich's theorem \cite[Theorem 5.6.1]{embeddingExemple1} that it is realizable over \(\mathbb{Q}\). By Galois correspondence, we can therefore find a Galois extension \(F/\mathbb{Q}\) such that \(\mathrm{Gal}(F/\mathbb{Q}) = G'^n\), where \(n\) is strictly greater than the number of subfields of \(LK\). This allows us to extract a subextension \(K'\) of \(F\) such that \(\mathrm{Gal}(K'/\mathbb{Q}) = G'\), and \(LK \cap K' = \mathbb{Q}\).
 Therefore, the compositum $K K'$ is a Galois extension of $\Q$ of group $G\times G'$ and $f_{p_L}(K K')\geq f_{p_L}(K)\geq 2$.
\end{proof}

%\begin{dfn}\label{def-spade-suit} Let $K$ be a number field and let  $(G_r)_{r\in \mathbb{N}}$ be countable family of finite groups. We say that the family \emph{has property ($\mathcal{DL}$) with respect to $K$} if for any finite extension $L/K$, there is a Galois extension  $L_r/K$ such that $\mathrm{Gal}(L_r/K) = G_r$, $L_r \cap L = K$, and $f_{p_L}(L_r)\geq 2$, where $p_L=\min\{\ell\in\mathcal{P}\setminus\{2\}\mid f_{\ell}(L)=1\}$. 
%\end{dfn}

The following is our obstruction to Property ($\mathfrak{S}$) over $\mathbb{Q}$.

\begin{prop}\label{generalCase}
 Let  $H = {G}\times{G'}$ such that ${G}$ is the direct product of an infinite family of non-trivial finite groups with property ($\mathcal{DL}$) and  ${G'}$ is a direct product of finite solvable groups. Then $G$ does not have property $(\mathfrak{S})$ over $\mathbb{Q}$.
  %  Let $\overline{G} = \prod_{r \ge 1} G_r$ be a direct product of finite groups, and let $K$ be a fixed number field. If $(G_r)_{r\geq 1}$ has property ($\mathcal{DL}$) with respect to $K$, then there exists a Galois extension $F/K$ such that $\mathrm{Gal}(M/K)=\overline{G}  $ and $\mathfrak{S}(M)<\infty$. \red{Moreover $M$ can be taken to be totally real if $K$ and the fields $L_r$ in Definition \ref{def-spade-suit} are so.} In particular $\overline{G} $ does not have property $(\mathfrak{S})$ with respect to $K$.
\end{prop}

\begin{proof} 
First, notice that we can write $H = \prod_{r \ge 1} G_r\times G'_r$ where, for every $r\geq 1$,  $G_r$ is a finite group with property $(\mathcal{DL})$ and $G'_r$ is a finite solvable (possibly trivial) group. By Lemma \ref{dir-prod-spade}, the group $H_r=G_r\times G'_r$ has also property ($\mathcal{DL}$).
Let $p_i$ denote the $i$-th odd prime number.

We  inductively construct a sequence $(M_r)_{r \ge 1}$ such that:
    \begin{enumerate}
        \item\label{item1-proof} $M_r/\Q$ is Galois and $\mathrm{Gal}(M_r/\Q)=H_r$.
        \item\label{item2-proof} For every $i \leq r$, $f_{p_i}(M_1 \ldots M_r)\geq 2$.
        \item\label{item3-proof} $M_1 \ldots M_{r-1} \cap M_r = \Q$.
    \end{enumerate}
For $r=1$, as $H_1$ satisfies property $(\mathcal{DL})$, taking $L=\Q$ in the definition of $(\mathcal{DL})$, there is a Galois extension $M_1/\Q$  of group $H_1$ and such that $f_{p_{1}}(M_1)=f_3(M_1)\geq 2$. So $M_1$ satisfies conditions \eqref{item1-proof}, \eqref{item2-proof} and \eqref{item3-proof}.
   
Suppose we have already constructed $M_1, \ldots, M_{r-1}$. Let \[i_r=\min\{i\in\mathbb{N}\mid f_{p_i}(M_1 \ldots M_{r-1})=1\}.\]  By the inductive hypothesis \eqref{item2-proof}, we necessarily have $i_r \ge r$. Set $L = M_1 \ldots M_{r-1}$, so that $p_{i_r}=p_L$. Since $H_r$ satisfies ($\mathcal{DL}$), there exists a Galois extension $M_r/\Q$ of group $H_r$ such that \eqref{item1-proof} and \eqref{item3-proof} hold. Now, let $i\leq r\leq i_r$. If $i<i_r$, then $f_{p_i}(M_1\cdots  M_r)\geq  f_{p_{i}}(M_1\cdots M_{r-1})\geq 2$ while, if $i=i_r$, $f_{p_{i}}(M_1\cdots  M_r)\geq  f_{p_{i_r}}(M_r)\geq 2$, so that \eqref{item3-proof} holds, and $f_{p_{L}}(M_{r})\ge 2$. Now, setting
    \[
    M = \prod_{r \ge 1} M_r ,
    \]
$M / \Q$ is Galois with Galois group $H = \prod_{r \ge 1} H_r$.

Let $p$ be an odd prime number. Then by \eqref{item2-proof}, for $r$ big enough, we have $f_{p}(M_1\cdots  M_r)\geq 2$. Therefore
by Lemma \ref{lemme elementaire} we have 
    \begin{align*}
        \mathfrak{S}(M) 
        \le\sum_{\substack{p \in \mathcal{P} \\ p \ne 2}} \min_{r \ge 1} \left( \mathfrak{S}_p(M_1\cdots M_r) \right)
        \le \sum_{\substack{p \in \mathcal{P} \\ p \ne 2}} \frac{\log(p)}{p^2 + 1} < +\infty.
    \end{align*}

Therefore, $H$ does not have property $(\mathfrak{S})$ over $\Q$.
\end{proof}

\section{Embedding Problems in Galois Theory and Scholz Extensions} \label{SectionEmbeddingProblem}
In this section, we recall some results on embedding problems in Galois theory (see \cite[Chapter 1]{embeddingExemple1} for more material on this topic).
%The literature on the embedding problem in Galois theory is extensive and has garnered attention from many researchers worldwide. For references on this topic, see for example \cite{embeddingExemple0}, \cite{embeddingExemple1}, \cite{embeddingExemple2}, \cite{embeddingExemple3}, \cite{embeddingExemple4}, \cite{embeddingExemple5}, and \cite{embeddingExemple6}. 

Let \( K \) be a fixed number field, and let \( G_K = \mathrm{Gal}(\overline{\mathbb{Q}}/K) \) be the absolute Galois group of \( K \).

Let \( L/K \) be a finite Galois extension such that $\mathrm{Gal}(L/K)$ is isomorphic to a quotient $G/N$ of a larger finite group $G$ having a normal subgroup $N$. The embedding problem associated to this data asks whether there exists a Galois extenion $M/L$ with $\mathrm{Gal}(M/L)=N$ and  such that $M/K$ is also Galois with $\mathrm{Gal}(M/K)=G$.

In other words, given the following exact sequence of groups
\begin{equation}\label{exact-seq-emb-pb}\tag{$\diamond$}
    1 \longrightarrow N \overset{i}{\hookrightarrow} G \overset{\pi}{\twoheadrightarrow} \mathrm{Gal}(L/K) \longrightarrow 1
\end{equation}
 the associated embedding problem asks for the existence of an epimorphism \( \varphi: G_K \rightarrow G \) such that \( \pi \circ \varphi = \psi \), where $\psi: G_K\rightarrow \mathrm{Gal}(L/K)$ is the canonical restriction map. 

%\begin{align*}  \label{embedding problems}
  %  &G_K \\
 % \tag{$\diamondsuit$}      &^{\varphi} \swarrow \downarrow  ^{_{\psi}}         \\
 %   1 \longrightarrow N \overset{\rho}{\hookrightarrow} G \overset{\pi}{\twoheadrightarrow}  \rm{Gal}&(L/K) \longrightarrow 1
%\end{align*}

%We say that \emph{the embedding problem (\ref{embedding problems}) has a weak solution} if there exists a homomorphism \( \varphi: G_K \rightarrow G \) such that 
%the following diagram commutes:
%\begin{align*}
%&G_K \\
%^{\varphi} \swarrow &\downarrow ^{\psi} \\
%G \overset{\pi}{\twoheadrightarrow} &\mathrm{Gal}(L/K)
%\end{align*}
%i.e., 

If such an epimorphism \( \varphi \) exists, it is called a \emph{(proper) solution} to the embedding problem.
In this case, setting \( M=K^{\operatorname{Ker}(\varphi)} \) the fixed field of \(\operatorname{Ker}(\varphi)\) and since \( L=K^{\operatorname{Ker}(\psi)} \), we have $L\subset M$
since \(\operatorname{Ker}(\varphi) \subset \operatorname{Ker}(\pi \circ \varphi) = \operatorname{Ker}(\psi)\).
Thus, passing to the quotient, \(\varphi\) induces an isomorphism $\mathrm{Gal}(M / K)\simeq  G$.

%The group $N$ is called the \emph{kernel} of the embedding problem.
%We  also say that the embedding problem is \emph{central} if \(N\) is contained in the center of \(G\). 

To solve certain embedding problems, we have a sufficient condition if the extension $L/K$ is of a special type, namely  a Scholz extension. 
Before defining such extensions,  we recall that, as usual, we denote by \( \mathcal{O}_K \) the ring of integers of a number field \( K \) and by $N_{K/\Q}:K\rightarrow \Q$ the norm function.
There are several definitions in the literature; here, we provide one that is clear and sufficient for our context, taken from \cite[p. 374]{ScholzField}.
\begin{dfn}
    Let \( p \) be a prime number and $n\geq 1$ an integer.  A Galois extension \( L/K \) of number fields is said to be \emph{\( n \)-Scholz for the prime  \( p \)},  if for any prime ideal \( \mathfrak{p}\subset \mathcal{O}_K \) that ramifies in \( L/K \), and \( \mathfrak{q} \subset \mathcal{O}_L \) with \( \mathfrak{q} \mid \mathfrak{p} \), the following three conditions hold:
    \begin{enumerate}
        \item\label{Scholz-1} \( \mathfrak{p} \nmid p \);
        \item\label{Scholz-2} \( N_{K/\mathbb{Q}}(\mathfrak{p}) \equiv 1 \pmod{p^{n}} \);
        \item\label{Scholz-3} the decomposition group $ D({\mathfrak{q}}/{\mathfrak{p}})$ equals the inertia group $ I({\mathfrak{q}}/{\mathfrak{p}}) $.
    \end{enumerate}
\end{dfn}

\begin{rmq}If \( K = \mathbb{Q} \), then condition \eqref{Scholz-1} is equivalent to the fact that  \( p \) does not ramify in \( L/\mathbb{Q} \) and \( L \) is totally real and condition \eqref{Scholz-2} can be restated by saying that if a prime $\ell$ ramifies in \( L/\mathbb{Q} \) then \( \ell\equiv 1\pmod{p^n} \). Moreover, condition \eqref{Scholz-3} is equivalent to the fact that the ramified primes in $L/\Q$ are totally ramified.
\end{rmq}

%If \( L/K \) is abelian, then \( D(L_{\mathfrak{q}}/K_{\mathfrak{p}}) \) and \( I(L_{\mathfrak{q}}/K_{\mathfrak{p}}) \) do not depend on the prime \( \mathfrak{q} \) above \( \mathfrak{p} \). In this case, condition $(3)$ can be restated as:

%The decomposition group and the inertia group of \( L/K \) over \( \mathfrak{p} \) are the same.

   %     Recall that the fixed field of the decomposition group is the largest subfield of \( L/K \) where \( \mathfrak{p} \) splits completely, and the fixed field of the inertia group is the largest subfield of \( L/K \) where \( \mathfrak{p} \) is unramified.

%Notice that, in general, by Hensel's lemma (\cite[Theorem 7.3]{CommutativeAlgebra}), condition \eqref{Scholz-2} is equivalent to the fact that the \( \mathfrak{p} \)-adic completion \( K_{\mathfrak{p}} \) of \( K \) contains all the \( p^n \)-th roots of unity. 

We have the following:
\begin{thm} (\cite[Proposition 10.3]{ScholzField}) \label{TheoremScholzExtensionSolveEmbeddingProblem}
    Let \( L/K \) be an \( n \)-Scholz extension of number fields for a rational prime number \( p \). Suppose we have an exact sequence \eqref{exact-seq-emb-pb} for some groups $G$ and $N$ such that $N\simeq \mathbb{Z}/p\mathbb{Z}$, $N$ is contained in the center of $G$ and the \( p \)-adic valuation of the exponent of \( G \) is at most \( n \). Then  the associated embedding problem has a solution.
\end{thm}
\section{ On certain groups with property $(\mathcal{DL})$}
\subsection{Property $(\mathcal{DL})$ for groups of odd order} \label{Subsection Grunwald solvable extension}
The first family we will examine consists of groups of odd order. This family differs from those studied later, as the arguments here rely on the Grunwald problem rather than the embedding problem.

\begin{prop} \label{PropositionGroupOddOrder}
Any finite, nontrivial group of odd order has property $(\mathcal{DL})$.
\end{prop}

\begin{proof}
Let $G$ be a finite group of odd order. Notice that $G$ is solvable, by the Feit-Thompson's theorem \cite[Theorem 1]{FeitThompson}.
    Let $L$ be any finite Galois extension over $\Q$, and let $p$ be the smallest odd prime with inertia degree of $1$ in $L$. 
 Denote $\mathrm{Sub}(L)$ the family of all number fields contained in $L$. For each $k \in \mathrm{Sub}(L) \setminus \mathbb{Q}$, choose $p_{k} \in \mathcal{P}$ such that the inertia degree of $p_{k}$ in $k/\mathbb{Q}$ is strictly greater than $1$, and set $S=\{p_{k}\}_{k \in \mathrm{Sub}(L)} \cup \{\infty\}$. Let $q$ be an odd prime that divides the order of $G$. By Cauchy's theorem (\cite[Theorem 3.2]{LangAlgebra}), we know that $\mathbb{Z}/q\mathbb{Z}$ is embeddable in $G$. 

    Consider the unique unramified extension $K_{p}$ of degree $q$ over the complete local field $\mathbb{Q}_{p}$, which we know is Galois with $\mathrm{Gal}(K_{p}/\mathbb{Q}_{p})=\mathbb{Z}/q\mathbb{Z}$ (see \cite[Theorem 2]{SerreLocalField}) . Since $2$ is relatively prime to the order of $G$,  and $\Q$ contains exactly $2$ root of unity, by applying \cite[Corollary 2]{GrunwaldSolvableCase} we obtain a Galois extension $K/\mathbb{Q}$ with Galois group $G$, such that every prime in $S$ splits totally in $K$, and the completion of $K$ with respect to the $p$-adic absolute value is exactly $K_{p}$. In particular, $K$ is totally real, and $p$ has an inertia degree greater than $q > 2$ in $K$.

    Now, consider $L \cap K$. If $L \cap K \in \mathrm{Sub}(L) \setminus \mathbb{Q}$, then $p_{L \cap K}$ splits completely in $L \cap K$ since it does in $K$. However, by definition, $p_{L \cap K}$ also has an inertia degree greater than $2$ in $L \cap K$, which is a contradiction. Hence, $L \cap K = \mathbb{Q}$, concluding the proof.
\end{proof}

\subsection{Property $(\mathcal{DL})$ for abelian groups }\label{SectionAbelianCase}

We begin this section by summarizing key results on the behavior of primes in quadratic extensions.

\begin{prop} \label{givenPrimesRemainsInQ(sqrt(q))}
    Let $p$ be a fixed odd prime and $m \ge 1$ an integer. Then
    \begin{itemize}
    \item There are infinitely many primes $q \equiv 1 \pmod{2^m}$ such that $p$ remains inert in $\mathbb{Q}(\sqrt{q})$.
    \item There are infinitely many primes $q \equiv 1 \pmod{2^m}$ such that $p$ splits in $\mathbb{Q}(\sqrt{q})$.
    \end{itemize}
    Moreover, if $q \equiv 1 \pmod{4}$ or $p \equiv 1 \pmod{4}$, then
    \begin{itemize}
        \item  $q$ remains inert in $\mathbb{Q}(\sqrt{p})$ if and only if $p$ remains inert in $\mathbb{Q}(\sqrt{q})$
        \item $q$ splits in $\mathbb{Q}(\sqrt{p})$ if and only if $p$ splits in $\mathbb{Q}(\sqrt{q})$.
    \end{itemize}
\end{prop}

\begin{proof}
    Let $a \in \mathbb{Z}$ be a non-zero quadratic residue modulo $p$. Choose an integer $t$ such that $t \equiv a \pmod{p}$ and $t \equiv 1 \pmod{2^m}$, possible by the Chinese Remainder Theorem since $p$ and $2^m$ are coprime. By Dirichlet’s theorem on primes in arithmetic progressions, we can pick a prime $q \equiv t \pmod{2^m p}$ to satisfy the desired property.

    Now, if $q \equiv 1 \pmod{4}$ or $p \equiv 1 \pmod{4}$, by the quadratic reciprocity law (\cite[Theorem 8.12]{PremierArithmetique}), we have
    $$
    \left(\frac{p}{q}\right) \left(\frac{q}{p}\right) = (-1)^{\frac{p-1}{2} \cdot \frac{q-1}{2}} = 1,
    $$
    so $\left(\frac{q}{p}\right) = 1$ if and only if $\left(\frac{p}{q}\right) = 1$. This completes the proof.
\end{proof}

We next establish the following:

\begin{prop}\label{PropositionZ/2^n_kZCase}
The group $\mathbb{Z}/2^{n} \mathbb{Z}$ has property $(\mathcal{DL})$ for every integer $n \geq 1$.
\end{prop}

\begin{proof}
    Let $L/\Q$ be a finite Galois extension, and let $p$ be the smallest odd prime such that $f_p(L) = 1$. By Proposition \ref{givenPrimesRemainsInQ(sqrt(q))}, we can find a prime $q \equiv 1 \pmod{2^{n}}$ such that $p$ remains inert in $\mathbb{Q}(\sqrt{q})$. Note that $\mathbb{Q}(\sqrt{q})/\mathbb{Q}$ is an $n$-Scholz extension, so by Theorem \ref{TheoremScholzExtensionSolveEmbeddingProblem}, there exists a cyclic $2^{n}$-extension $M/\mathbb{Q}$ containing $\Q(\sqrt{q})$ such that the intersection $L \cap M$ is either trivial or contains $\mathbb{Q}(\sqrt{q})$.

    Since $p$ remains inert in $\mathbb{Q}(\sqrt{q})$ and has inertia degree $1$ over $L$, it follows that $L \cap M = \mathbb{Q}$. Thus, we conclude the proof.
\end{proof}

\begin{coro} \label{Corollary Abelian group}
    Every nontrivial, abelian group satisfies property $(\mathcal{DL})$.
\end{coro}
\begin{proof}
    This follows directly from Lemma \ref{dir-prod-spade}, Proposition \ref{PropositionGroupOddOrder}, Proposition \ref{PropositionZ/2^n_kZCase}, and the classification of finite abelian groups.
\end{proof}

\subsection{Property $(\mathcal{DL})$ for Hamiltonian groups and Dihedral group of order 8}\label{Section Hamiltonian Case}
Let $\mathbb{H}_{8}$ denote the quaternion group of order $8$ and $\mathbb{D}_{8}$ the dihedral group of order $8$. Both groups are solvable and non-abelian, and the proof also relies on embedding problems.

\begin{prop} \label{H8etD8case}
    The groups $\mathbb{H}_{8}$ and $\mathbb{D}_{8}$ have property $(\mathcal{DL})$.
\end{prop}

\begin{proof}
    Let $L/\Q$ be a finite Galois extension, and let $p$ be the smallest odd prime with inertia degree $1$ in $L$. By Proposition \ref{givenPrimesRemainsInQ(sqrt(q))}, we can choose a prime $q \equiv 1 \pmod{4}$ such that $p$ remains inert in $\mathbb{Q}(\sqrt{q})$. We then select another prime $l \equiv 1 \pmod{4}$ such that $q$ splits in $\mathbb{Q}(\sqrt{l})$, and $\sqrt{l}, \sqrt{ql} \notin L$, possible by the infinitude of such primes in Proposition \ref{givenPrimesRemainsInQ(sqrt(q))}.

    Since $q \equiv 1 \pmod{4}$, $l$ also splits in $\mathbb{Q}(\sqrt{q})$, so $\mathbb{Q}(\sqrt{q}, \sqrt{l}) / \mathbb{Q}$ is a 2-Scholz extension. By Theorem \ref{TheoremScholzExtensionSolveEmbeddingProblem}, we can solve the central embedding problems with cyclic kernels for the extension $\mathbb{Q}(\sqrt{q}, \sqrt{l}) / \mathbb{Q}$.
    This gives us two extensions, $M_{1}$ and $M_{2}$, where $M_{1}$ is a quaternion extension of order $8$ and $M_{2}$ is a dihedral extension of order $8$, both containing $\mathbb{Q}(\sqrt{q}, \sqrt{l})$.

    Since $M_{1}$ has a unique sub-quartic extension, $\mathbb{Q}(\sqrt{q}, \sqrt{l})$, and three sub-quadratic extensions, $\mathbb{Q}(\sqrt{q})$, $\mathbb{Q}(\sqrt{l})$, and $\mathbb{Q}(\sqrt{ql})$, any subfield of $M_{1}$ must contain one of these fields. However, by construction, $\sqrt{l}, \sqrt{ql} \notin L$, and $\sqrt{q} \notin L$ since $p$ is inert in $\mathbb{Q}(\sqrt{q})$. Thus, $L \cap M_{1} = \mathbb{Q}$.

    The same reasoning applies to $M_{2}$, concluding the proof.
\end{proof}

\begin{coro} \label{Corollary Hamiltonian groups}
    Every Hamiltonian group satisfies property $(\mathcal{DL})$.
\end{coro}
\begin{proof}
    This follows directly from Lemma \ref{dir-prod-spade},  Corollary \ref{Corollary Abelian group}, Proposition \ref{H8etD8case}, and the classification of Hamiltonian groups by Dedekind and Baer (see \cite[5.3.7]{classificationHalmitoniangroup}).
\end{proof}

\subsection{Property $(\mathcal{DL})$ for the Symmetric group} \label{Section Symmetric Case}

 In this case we also study the splitting behavior of primes in quadratic extensions and use results on embedding problems.

\begin{prop} \label{SnCase}
For every integer \( n \geq 2 \), the symmetric group \( \mathfrak{S}_{n} \) has property $(\mathcal{DL})$.
\end{prop}

\begin{proof}
    Let \( i_{k} > 2 \), \( L \) be a finite Galois extension over $\Q$, and \( p \) the smallest odd prime with inertia degree 1 in \( L \). Consider a quadratic extension \( F/\mathbb{Q} \) where \( p \) remains inert. By a specific case of the embedding problem \cite[Theorem 2.7]{SnPasPropSgothique}, we can construct, an \( \mathfrak{S}_{i_{k}} \)-extension  \( M/\mathbb{Q} \), with \( F \subset M \). Assume \( i_{k} \neq 4 \).

    The field \( L \cap M \), as a Galois subfield of \( M \), must equal \( \mathbb{Q} \), \( F \), or \( M \), since \( M \) is an \( \mathfrak{S}_{i_{k}} \)-extension and the only normal subgroups of \( \mathfrak{S}_{i_{k}} \) (for \( i_{k} \neq 4 \)) are \( \mathfrak{S}_{i_{k}} \), \( \mathfrak{A}_{i_{k}} \), and \( \{ 1_{\mathfrak{S}_{i_{k}}} \} \). Since \( p \) is inert in \( F \) and has inertia degree 1 in \( L \), it follows that \( L \cap M = \mathbb{Q} \).

    For \( i_{k} = 4 \), \( \mathfrak{S}_{4} \) has an additional normal subgroup, the Klein four-group, meaning \( [L \cap M : \mathbb{Q}] = 6 \) could be possible. However, \( (L \cap M)/\mathbb{Q} \), being Galois, would then have a unique sub-quadratic extension \( F \), which is a contradiction. Thus, \( [L \cap M : \mathbb{Q}] = 6 \) cannot occur, completing the proof.
\end{proof}

\subsection{Property $(\mathcal{DL})$ for groups with cyclic \(2\)-Sylow subgroups} \label{a direct product of group having a cyclic 2-Sylow}

We begin with a classical result, which is a specific case of a more general theorem in \cite[6.2.11]{GroupTheoryScott}.

\begin{prop}\label{Proposition2-sylowcyclicsolvable}
    Let \(G\) be a finite group of order \(2^jq\) with \(j \ge 1\) and \(q\) odd, having a cyclic \(2\)-Sylow subgroup. Then \(G\) has a unique normal subgroup of order \(q\) and is therefore solvable.
\end{prop}
We then have the following: 
\begin{prop}
   Every finite group with a cyclic \(2\)-Sylow subgroup has property \((\mathcal{DL})\).
\end{prop}

\begin{proof}
    Let \(L/\Q\) be a finite Galois extension and \(p\) the smallest odd prime with inertia degree 1 in \(L\). For a fixed integer \(l\), let \(G_{l}\) be a group with a cyclic \(2\)-Sylow subgroup \(\mathbb{Z}/2^{j_{l}}\mathbb{Z}\) and (Proposition \ref{Proposition2-sylowcyclicsolvable} ) unique normal subgroup \(H_{l}\) of odd order.
 Using Proposition \ref{givenPrimesRemainsInQ(sqrt(q))}, we select \(n = |SF(L)|\) quadratic fields \(F_{1}, \ldots, F_{n}\) such that \(p\) remains inert in each and \(F_{i} \cap F_{k} = \mathbb{Q}\) for \(i \neq k\). For each \(F_k\), by proceeding as in the proof of Proposition \ref{PropositionZ/2^n_kZCase}, we can construct a cyclic \(2^{j_{l}}\)-extension \(M_{k}\) where \(F_{k}\) is the unique sub-quadratic field. Thus, \(M_{1} \cdots M_{n}\) is Galois over \(\mathbb{Q}\) with Galois group \(\left(\mathbb{Z}/2^{j_{l}}\mathbb{Z}\right)^{n}\).
    Since \(\prod_{k=1}^{n} H_{l}\) is solvable and has order coprime to \(\mathrm{Gal}(M_{1} \cdots M_{n} / \mathbb{Q})\), by \cite[Theorem 5.6.2]{embeddingExemple1} we can find a Galois extension \(N/\mathbb{Q}\) with \(\mathrm{Gal}(N/\mathbb{Q}) = \prod_{k=1}^{n} G_{l}\) and \(M_{1} \cdots M_{n} \subset N\).

    Let \( \Delta_{1}, \ldots, \Delta_{n} \) be \(G_{l}\)-extensions of \(\mathbb{Q}\) within \(N\) such that they are linearly disjoint and \(M_{k} \subset \Delta_{k}\) for each \(k\). For \(i \neq j\), we have:
    \[
    \Delta_{i} \cap L \neq \Delta_{j} \cap L \quad \text{or} \quad \Delta_{i} \cap L = \Delta_{j} \cap L = \mathbb{Q}.
    \]
    Suppose \(\Delta_{i} \cap L \neq \mathbb{Q}\) for all \(i\). Then, the map \(i \mapsto \Delta_{i} \cap L\) would be injective and bijective since \(n = |SF(L)|\), implying that \(\mathbb{Q}\) has a preimage, which is a contradiction. Therefore, there exists an \(i\) such that \(\Delta_{i} \cap L = \mathbb{Q}\).

    Since \(p\) is inert in \(F_{i} \subset M_{i} \subset \Delta_{i}\), it has an inertia degree greater than 2 in \(\Delta_{i}\). Thus, \(\mathrm{Gal}(\Delta_{i}/\mathbb{Q}) = G_{l}\), \(\Delta_{i} \cap L = \mathbb{Q}\), and \(p\) have an inertia degree greater than 2 in \(\Delta_{i}\), completing the proof.
\end{proof}

\subsection{Proof of Theorem \ref{Theorem Negative answer to 2nd problem} and some remarks} \label{Section negative answer to 2nd problem}
\begin{proof}[Proof of Theorem \ref{Theorem Negative answer to 2nd problem}]
   Let $G = \prod_{i=1}^{\infty} G_{i}$ be a direct product of finite solvable groups or symmetric groups, as described in Theorem \ref{Theorem Negative answer to 2nd problem}. Depending on the condition satisfied by $G$, and with Proposition \ref{generalCase}, we conclude as follows:
\begin{itemize}
    \item If $G$ satisfies condition (\ref{Groups of odd order}), then Proposition \ref{PropositionGroupOddOrder} applies.
    \item If $G$ satisfies condition (\ref{Abelian groups}), then Corollary \ref{Corollary Abelian group} applies.
    \item If $G$ satisfies condition (\ref{Hamiltonian groups}), then Corollary \ref{Corollary Hamiltonian groups} applies.
    \item If $G$ satisfies condition (\ref{Symmetric groups}), then Proposition \ref{SnCase} applies.
    \item If $G$ satisfies condition (\ref{Groups of even order having a cyclic 2-Sylow subgroup}), then Proposition \ref{Proposition2-sylowcyclicsolvable} applies.
\end{itemize}

\end{proof}

\begin{rmq}

Following the method employed in the proof of Proposition \ref{PropositionGroupOddOrder}, it appears plausible that a result analogous to Theorem \ref{Theorem Negative answer to 2nd problem} could be established by relying on a different variant of the inverse Galois problem, namely the Grunwald problem, which has already been discussed in this article. This would concern groups for which the Grunwald problem admits a positive solution, at least for cyclic extensions over $p$-adic fields. Notable examples include groups admitting a generic Galois extension over \(\mathbb{Q}\) (see \cite[Theorem 5.9]{GenericGaloisExtension}), such as the alternating group \(\mathfrak{A}_5\) (\cite[Theorem 5.1]{GenericGaloisExtension} and \cite[Theorem]{NoetherProblemAlternatingGroupA5Maeda}), as well as the symmetric groups.

With appropriate adjustments and refinements to the construction used in Proposition \ref{PropositionGroupOddOrder}, it seems reasonable to expect that the result could extend to groups for which the \emph{tame} Grunwald problem holds, namely, when the problem is restricted to places not dividing the order of the group \(G\). This version, which avoids the Brauer–Manin obstruction, includes examples such as the quaternion group, the dihedral group of order 8, and more generally, supersolvable groups \cite[Corollaire p.3]{GrunwaldProblemSuperSolvableGroup}.

In particular, if the \emph{tame} Grunwald problem has a solution for all solvable groups, it would likely imply Theorem \ref{Theorem Negative answer to 2nd problem} for this class of groups. This observation supports the heuristic expectation that the general answer to Question \ref{MainQuestion} might be negative.

\end{rmq}

\section{Some final remarks on property ($\mathfrak{S}$)}
\subsection{On the intersection of three criteria for the Northcott property}
\label{Subsection On the Intersection of Three Criteria}
In this subsection, we aim to compare Criterion \ref{Criteria (S) imply (N)}, Bombieri and Zannier's criterion for the Northcott property \cite[Theorem 1]{bombierizanniernorthcottproperty}, and Widmer's criterion for the Northcott property \cite[Theorem 3]{widmernorthcottcriteria}. We recall that \( K^{(d)}_{ab} \) denote the compositum of all abelian extensions of degree at most \( d \) of a number field \( K \). While we have already discussed Property ($\mathfrak{S}$), let us briefly recall the other two criteria.

\begin{thm}$($Bombieri and Zannier \cite[Theorem 1]{bombierizanniernorthcottproperty}$)$ \label{Theorem Criterion BZ}
   Let $K$ be a number field. Then, property (N) holds for the field $K^{(d)}_{ab}$ 
, for any $d\ge 2$.
\end{thm}
   The next one, which is Widmer's criterion, is based on
 the growth of certain discriminants in infinite towers of number fields. 
\begin{thm}$($Widmer \cite[Theorem 1.3]{widmernorthcottcriteria}$)$ \label{Theorem criterion Widmer}
    Let $K$ be a number field. Let $K=K_0 \subsetneq K_1 \subsetneq K_2 \subsetneq$.... be a nested sequence of finite extensions and set $L=\bigcup_i K_i$. Suppose that

\begin{equation}\label{Widmer criterion}\inf _{K_{i-1} \subsetneq M \subseteq K_i}\left(N_{K_{i-1} / \mathbb{Q}}\left(D_{M / K_{i-1}}\right)\right)^{\frac{1}{\left[M: K_{0}\right]\left[M: K_{i-1}\right]}} \longrightarrow \infty
\end{equation}

as $i$ tends to infinity where the infimum is taken over all intermediate fields $M$ strictly larger than $K_{i-1}$, $ D_{M / K_{i-1}}$ denotes the relative discriminant of the extension $M / K_{i-1}$ and $N_{K_{i-1} / \mathbb{Q}}\left(D_{M / K_{i-1}}\right)$ denotes the unique positive integer generating the norm ideal. Then the field $L$ has the Northcott property.
\end{thm}

Examples of infinite extensions of $\mathbb{Q}$ satisfying one of the three criteria but not the others have already been constructed. For instance, in \cite[Theorem 1.5]{checcolinorthcottproperty}, an infinite Galois extension of the rationals is constructed that satisfies property ($\mathfrak{S}$), but does not satisfy condition (\ref{Widmer criterion}) of Theorem \ref{Theorem criterion Widmer} and is not contained in any $K^{(d)}_{ab}$ for any number field $K$ and $d \geq 2$.
One might ask whether only finite Galois extensions of the rationals satisfy all three conditions simultaneously. However, here we provide an infinite Galois extension of the rationals that satisfies all three conditions.

\begin{prop} \label{Proposition Field with all criterion}
    There exists an infinite Galois extension $L/\Q$ contained in $\Q_{ab}^{(2)}$, satisfying property ($\mathfrak{S}$) and condition (\ref{Widmer criterion}) of Theorem \ref{Theorem criterion Widmer}.
\end{prop}

\begin{proof}
    Let set $n_{1}=1$ and $p_{1}=2$. We construct inductively a sequence of integers \( (n_i)_{i \geq 1} \) and an increasing sequence of prime numbers \( (p_i)_{i \geq 1} \) such that: \begin{itemize}
        \item Every prime \( p \leq n_i \) splits totally in \( \Q(\sqrt{p_i}) \),
        \item  \(\sum_{n_{i-1} < p \leq n_i} \mathfrak{S}_p(\Q(\sqrt{p_1}, \ldots, \sqrt{p_{i-1}})) \geq 1.\)
    \end{itemize} 

    Assume \( p_1, \ldots, p_{i-1} \) and \( n_1, \ldots, n_{i-1} \) have already been constructed. By Proposition \ref{numbers field case}, we can choose \( n_i \) such that:  
    \[
    \sum_{n_{i-1} < p \leq n_i} \mathfrak{S}_p(\Q(\sqrt{p_1}, \ldots, \sqrt{p_{i-1}})) \geq 1.
    \]  

     Let $a= \prod_{2 \le q \le n_{i}}q$ and consider the arithmetic progression $1,a+1,2a+1, \ldots$. By Dirichlet's theorem on arithmetic progressions, we can choose a prime \( p_i \) such that \( p_i > \max(n_i, p_{i-1}) \) and \( p_i \equiv 1 \mod a \). In particular, every prime \( 2 \leq q \leq n_i \) is a square in \( \mathbb{Z}/p_i\mathbb{Z} \), so it splits totally in \( \Q(\sqrt{p_i}) \).

We claim that the field \( L = \mathbb{Q}((\sqrt{p_i})_{i \geq 1}) \) satisfies the proposition.

Indeed, as a compositum of quadratic extensions, it is clear that \( L \subset \mathbb{Q}_{\mathrm{ab}}^{(2)} \). Additionally, property \( (\mathfrak{S}) \) is satisfied due to the conditions imposed during its construction.

To verify this, define \( L_i = \mathbb{Q}(\sqrt{p_1}, \ldots, \sqrt{p_i}) \). We compute:
\begin{align*}
\mathfrak{S}(L) &= \sum_{p \in \mathcal{P}} \mathfrak{S}_p(L) = \sum_{i=2}^{+\infty} \sum_{n_{i-1} < p \leq n_i} \mathfrak{S}_p(L) \\ 
&= \sum_{i=2}^{+\infty} \sum_{n_{i-1} < p \leq n_i} \mathfrak{S}_p(L_{i-1}) 
\geq \sum_{i=2}^{+\infty} 1 = \infty.
\end{align*}

Finally, \( L \) satisfies condition (\ref{Widmer criterion}) of Theorem \ref{Theorem criterion Widmer}. For simplicity, we assume that \( p_i \equiv 1 \pmod{4} \) for all \( i \geq 2 \), so that the discriminant of \( \mathbb{Q}(\sqrt{p_i}) \) is exactly \( p_i \). By \cite[Corollary (2.10) III]{NeukirchDiscriminantRelatif} and \cite[Theorem 88]{HilbertDavid}, we have for all \( i \geq 2 \):
\[
N_{L_i/\mathbb{Q}}(D_{L_{i+1}/L_i}) = \frac{D_{L_{i+1}/\mathbb{Q}}}{D_{L_i/\mathbb{Q}}^{[L_{i+1} : L_i]}} 
= \frac{D_{L_i/\mathbb{Q}}^2 D_{\mathbb{Q}(\sqrt{p_{i+1}})/\mathbb{Q}}^{[L_i : \mathbb{Q}]}}{D_{L_i/\mathbb{Q}}^{[L_{i+1} : L_i]}} 
= D_{\mathbb{Q}(\sqrt{p_{i+1}})/\mathbb{Q}}^{[L_i : \mathbb{Q}]} 
= p_{i+1}^{[L_i : \mathbb{Q}]}.
\]

Thus, we have:
\[
N_{L_i/\mathbb{Q}}(D_{L_{i+1}/L_i})^{\frac{1}{[L_{i+1} : \mathbb{Q}][L_{i+1} : L_i]}} 
= p_{i+1}^{\frac{[L_i : \mathbb{Q}]}{[L_{i+1} : \mathbb{Q}][L_{i+1} : L_i]}} 
= p_{i+1}^{\frac{1}{[L_{i+1} : L_i]^2}} 
= p_{i+1}^{\frac{1}{4}},
\]
As \( p_{i+1}^{\frac{1}{4}} \to \infty \) as \( i \to \infty \), we conclude the proof.

\end{proof}

\subsection{On Northcott Numbers}
Following the work of Robinson \cite{RobinsonDecisionproblem}, the concept of the Northcott number for a set of algebraic numbers was introduced by Vidaux and Videla \cite[p. 59]{NortcottNumberIntroduction}, with the aim of studying certain decidability problems in rings of integers. Among other results, they demonstrated that any subring of a ring of totally real integers possessing the Northcott property has an undecidable first-order theory \cite[Theorem 2]{NortcottNumberIntroduction}.  

The \emph{Northcott number} of a set of algebraic numbers \(X\) is denoted by  
\[
\mathcal{N}_X = \inf \left\{ t \in [0,+\infty)\; | \; \# \{\alpha \in X \; | \; h(\alpha)<t \}=\infty \right\}.
\]

They posed the following natural question:  
\begin{Question} \label{Question Northcott Number} $($\cite[Question 6]{NortcottNumberIntroduction}$)$  
Which real numbers can be realized as Northcott numbers of ring extensions of \(\mathbb{Q}\)?  
\end{Question}  

It is straightforward to see that \(0\) is a Northcott number (taking \(X = \overline{\mathbb{Q}}\)), as is \(\infty\) (taking \(X\) to be any set with the Northcott property, such as a number field).  

In \cite[Theorem 3]{NorthcottNumberPazuki}, it was shown that for any \(t \geq 0\), there exists a field \(L \subset \overline{\mathbb{Q}}\) satisfying  
\[
t \leq \mathcal{N}_h(L) \leq \mathcal{N}_h\left(\mathcal{O}_L\right) \leq 2t.
\]
The fields in their constructions are of the form \(\mathbb{Q}\left(p_1^{1/d_1}, p_2^{1/d_2}, p_3^{1/d_3}, \ldots\right)\) where the $p_{i}$ are prime numbers and the $d_{i}$ are suitably chosen integers. Their method leverages Widmer's criterion for the Northcott property \cite[Theorem 3]{widmernorthcottcriteria}, which is particularly easy to apply for fields of this form (as already used in the proof of Proposition \ref{Proposition Field with all criterion}).  

In this section, we establish a \emph{weak version} of this result but which does not rely on Widmer's criterion for the Northcott property. Instead, we use property (\(\mathfrak{S}\)) and, more specifically, the following result, which is a direct consequence of \cite[Theorem 2 and Example 2]{bombierizanniernorthcottproperty}:  

\begin{thm} \label{Theorem Bombieri Zannier}  
For any finite set of primes \(\mathcal{S}\), let  
$\mathbb{Q}^{t_{\mathcal{S}}}$
be the maximal subfield of \(\overline{\mathbb{Q}}\) where all primes in \(\mathcal{S}\) split totally. Then,  
\[
\frac{1}{2} \sum_{p \in \mathcal{S}} \frac{\log(p)}{p+1} \leq \underset{\alpha \in \mathbb{Q}^{t_{\mathcal{S}}}}{\liminf}\, h(\alpha) = \mathcal{N}_{\mathbb{Q}^{t_{\mathcal{S}}}} \leq \sum_{p \in \mathcal{S}} \frac{\log(p)}{p-1},
\]  where the $\liminf$ is taking with respect to set inclusion.
\end{thm}  

This approach provides a different method compared to that of \cite[Theorem 3]{NorthcottNumberPazuki}, and the fields constructed through this method, achieving a Northcott number in the desire range, are of an entirely different nature. Specifically, they are determined by local properties.

\begin{prop}
    Let $r>0$ be a real number. For every real $\epsilon>0$, there exists a finite set $\mathcal{S}$ of prime numbers such that $$r-\epsilon<\mathcal{N}_{\mathbb{Q}^{t_{\mathcal{S}}}}\le 2r.$$
\end{prop}

\begin{proof}
   In view of applying Theorem \ref{Theorem Bombieri Zannier}, we aim to choose \(\mathcal{S}\) so that both sums that appears in the statement lie in \([r - \epsilon, 2r]\). Let \((p_i)_{i \geq 1}\) denote the increasing sequence of all primes. Observe that:  
\[
\sum_{i\ge 1} \frac{\log p_{i}}{p_{i}-1} - \frac{\log p_{i}}{p_{i}+1} = \sum_{i \ge 1} \frac{2 \log p_{i}}{(p_{i}-1)(p_{i}+1)} < +\infty.
\]  

Thus, we can choose an integer \(l\) such that:  
\begin{equation} \label{somme}
\sum_{i \geq l} \left( \frac{\log(p_i)}{p_i-1} - \frac{\log(p_i)}{p_i+1} \right) \leq \epsilon.
\end{equation}

Select an integer \(j\) such that:  
\begin{equation} \label{terme}
\frac{\log p_j}{p_j+1} < \epsilon.
\end{equation} 

Set \(c = \max(j, l)\), and let \(j_0\) be the largest integer such that:  
\[
\sum_{i = c}^{j_0} \frac{\log p_i}{p_i-1} \leq 2r.
\]  

Now, by \eqref{somme}, \eqref{terme} and the definition of $j_{0}$ we have:  
\[
2r < \sum_{i = c}^{j_0+1} \frac{\log(p_i)}{p_i-1} \leq \sum_{i = c}^{j_0} \frac{\log(p_i)}{p_i+1} + 2\epsilon.
\]  

Thus:  
\[
r - \epsilon < \frac{1}{2} \sum_{i = c}^{j_0} \frac{\log(p_i)}{p_i+1}.
\]  

The field \(\Q^{t_{\mathcal{S}}}\), with \(\mathcal{S} = \{p_c, \ldots, p_{j_0}\}\), has the desired Northcott number according to Theorem \ref{Theorem Bombieri Zannier}.
\end{proof}

It is noteworthy that Question \ref{Question Northcott Number} has been answered completely by Okazaki and Sano \cite[Theorem 3]{NorthcottNumberAnswer}, using a refined version of the construction in \cite[Theorem 3]{NorthcottNumberPazuki}.

\section*{Acknowledgment}
I am very grateful to the anonymous referee for their careful reading of a previous version of this article and for their insightful remarks. In particular, I thank the referee for pointing out the interesting possible application of property~$(\mathfrak{S})$ mentioned in Remark~\ref{Remark Referee Question}.

\bibliographystyle{alpha}
\bibliography{bibliography}

\end{document}